%% file: PoincareDuality.tex
\documentclass[12pt]{article}
\pdfoutput=1
\usepackage[margin=1in]{geometry}  
\usepackage{graphicx}              
\usepackage{amsmath}               
\usepackage{amsfonts}              
\usepackage{amsthm}                
\usepackage{color}
\usepackage{tabulary}
\usepackage{float}

\newtheorem{thm}{Theorem}[section]
\newtheorem{lemma}[thm]{Lemma}
\newtheorem{prop}[thm]{Proposition}
\newtheorem{cor}[thm]{Corollary}

\DeclareMathOperator{\Ker}{Ker}
\DeclareMathOperator{\im}{Im}

\title{Visualizing Poincar\'{e} Duality}
\author{Lucien Clavier}

\begin{document}
\bibliographystyle{plain}
\maketitle
\tableofcontents

\pagebreak
\section{Introduction}
\label{introduction}
\input{introduction}

\pagebreak
\section{Dual cellulation and duality}
\label{cellulation}
\input{cellulation}

\pagebreak
\section{Cap product and duality}
\label{cap}
\input{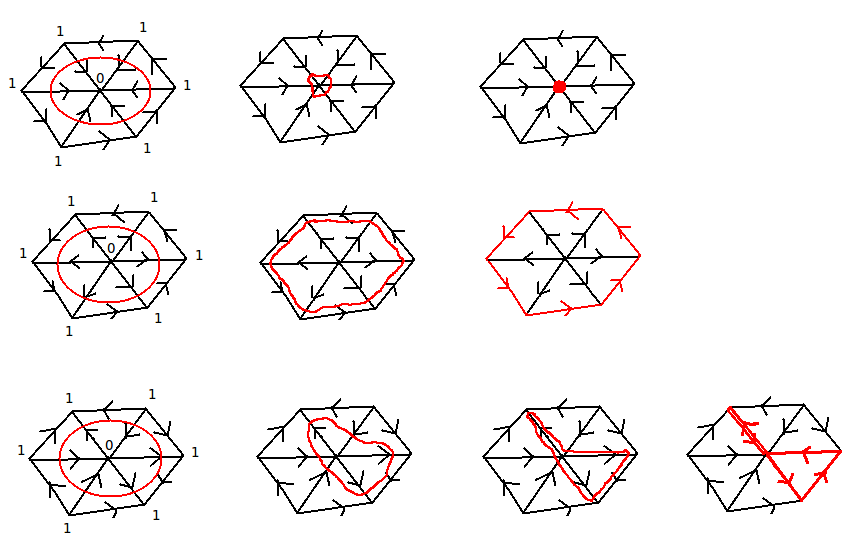}

\pagebreak
\section{Construction of level curves}
\label{levelcurves}
\input{levelcurves}

\pagebreak
\section{Deforming level curves}
\label{deform}
\input{deform}

\pagebreak
\section{Generalization to higher dimensions}
\label{generalization}
\input{generalization}

\pagebreak
\section{$D_3$ is an isomorphism}
\label{isom}
\input{isom}

\pagebreak
\section{Non-orientable case}
\label{nonor}
\input{nonor}

\bibliography{bibliography}		
\end{document}

%% file: introduction
For any orientable topological manifold $ M $ of dimension $ n $,  $ H^p(M) $ and $ H^c_{n-p} (M) $ are always isomorphic; this fact is called Poincar\'{e} duality.

There are two ways to define an actual isomorphism between the two:

The first one, involving dual cellulation, is more visual, and appeared earlier in History. It is defined in section \ref{cellulation}.

The second, more modern since directly defined in the context of singular (co)-homology, has the drawback of being manipulated only in algebraic terms; as a result, the reason why it does work is not transparent, from a geometric point of view. See section \ref{cap}.

Another link between cohomology and homology, in dimension 2, is to associate to each 1-cocycle some ``level curve''. This process is explained in section \ref{levelcurves}.

\mbox{}

Our goal is both to connect these three notions, and to provide a natural visualization of them.
We will do that in the remaining sections.

%% file: cellulation
We are going to consider only the restricted case where $M$ is a compact, orientable, triangulated manifold, of dimension $n$ and finite triangulation $\Delta$. The (co)-homology we are considering is the simplicial one, with coefficients in $\mathbb{Z}$.

Orientability here means that there is a way of assigning a sign $\pm 1$ to each $n$-simplex of $\Delta$ in such a way that the sum of all such signed simplices is a $n$-cycle, see \cite{Bredon}, for instance. 
The homology class of such a $n$-cycle generates the $n^\text{th}$ homology group, and is called a fundamental class for $M$.

We are going to consider first the dual cellulation defined on each $n$-simplex. 

Let $T=[v_0, \ldots, v_n]$ be a $n$-simplex of $\Delta$, and let $S=[v_{i_0}, \ldots, v_{i_k}]$ be some $k$-face of $T$, where as usual $i_1< \ldots <i_k$.

The dual cell associated to $S$ in $T$ is the convex hull of all barycenters of subsets of vertices of $T$ that contain the vertices of $S$. It is part of the barycentric subdivision of $\Delta$, and therefore inherit a canonical orientation.
We will choose the convention that any $p$-face of the barycentric subdivision of $\Delta$ is written $[b_1, \ldots, b_p]$, where every $b_j$ is the barycenter of some set of vertices containing less elements that the set of vertices for $b_{j+1}$.
See \cite{Hatcher} (where the orientation convention is reversed from ours) for details about barycentric subdivision.

\mbox{}

For $n=3$, we get the following picture, where orientation is omitted:

\begin{figure}[H]
\centering
\includegraphics[width=80mm]{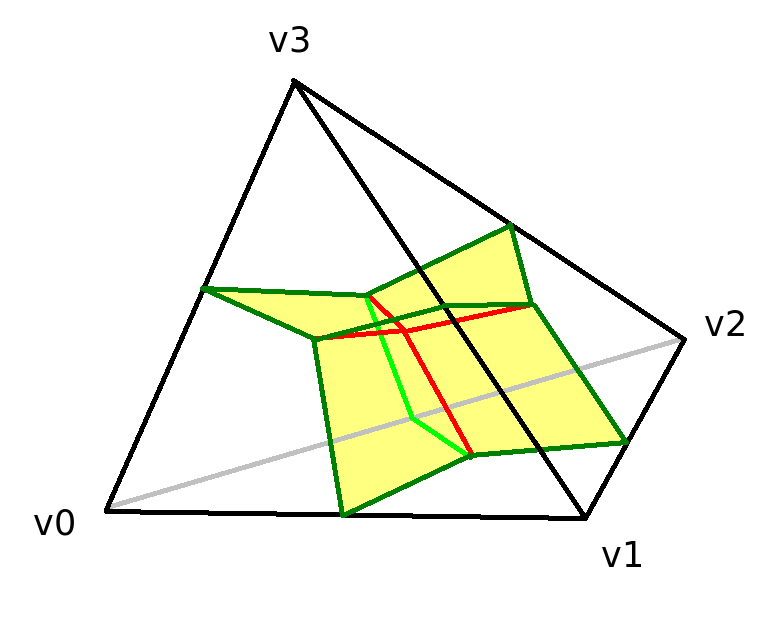}
\label{highdim22}
\end{figure}

The dual cell associated to a $2$-face of $T$ is the edge connecting the center of $T$ to the center of that face (in red in the picture); the dual cell associated to an edge of $T$ is the dimension-2 convex hull of the center of $T$, the center of this edge and the center of the two 2-faces containing this edge (in yellow in the picture). The dual cell of $T$ is its center, and the dual of this center is $T$ itself.

\mbox{}

Now, choose some $k$-cochain $\varphi$ of $M$. We associate to it the $(n-k)$-chain $D_1(\varphi)$ consisting of the sum of the dual of every $k$-simplex $S$, counted $\varphi(S)$ times.

More precisely and taking orientation into account, for every $n$-simplex $T$ and every of its $k$-face $S$, write $S^\ast_T$ the dual of $S$ in $T$. Then for every $k$-chain $\varphi$,
\[
D_1(\varphi)=
\sum_{\substack{
   T \text{ a }n\text{-simplex of }\Delta\\
  S \text{ a }k\text{-face of }T
  }}
 \epsilon_1(T) \epsilon_2(S,T) \varphi(S) \cdot S^\ast_T 
\]
where $\epsilon_1(T)$ is the sign of $T$ in some fundamental class of $M$, and $\epsilon_2(S,T)$ is the sign of $S$ in $\partial^{n-k}T$.

We leave to the reader to see why these signs give consistent orientations of dual cells when we glue every contiguous $n$-simplices together.

\mbox{}

In dimension 2, here are the different orientations involved:
\begin{figure}[H]
\centering
\includegraphics[width=115mm]{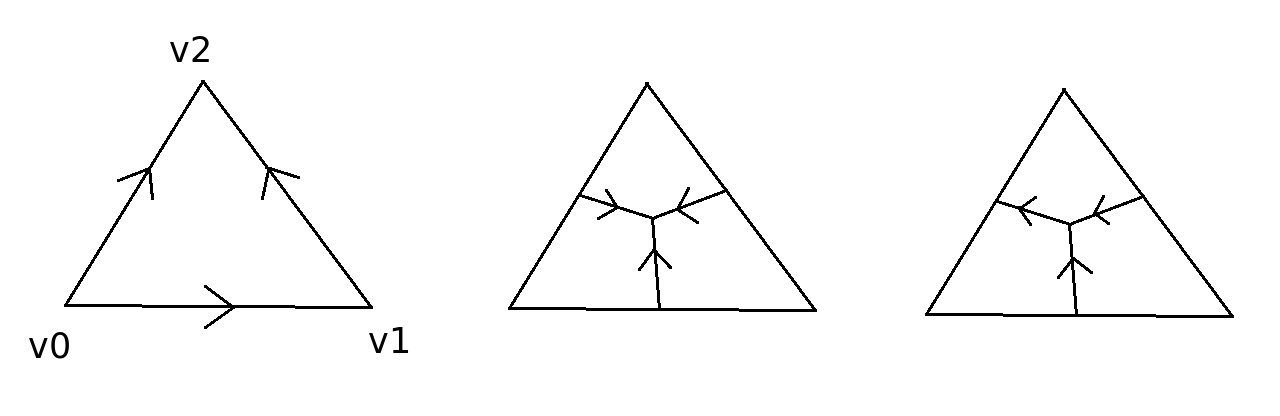}
\label{orcell}
\caption{The standard 2-simplex, with (from left to right) the standard orientation, the orientation of the dual 1-cells, and the same cells with coefficient $\epsilon_2([v_i,v_j])$}
\end{figure}

To finish this section, here are two examples of 1-cochains (first column), and their image under $D_1$ (second column).

\begin{figure}[H]
\centering
\includegraphics[width=90mm]{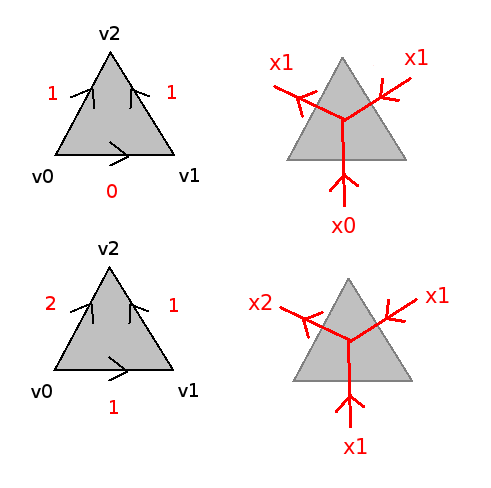}
\label{r1}
\end{figure}

%% file: cap
In the same setting as before (namely for $M$ a compact, orientable, triangulated manifold, of dimension $n$ and triangulation $\Delta$), consider the cap product $ \cap : C^k\times C_n \rightarrow C_{n-k}$ defined by
\[
(\varphi,[v_0,\ldots,v_n]) \overset{\cap}{\longmapsto} \varphi([v_0,\ldots,v_k])\cdot [v_k,\ldots,v_n]
\]
where as usual, $C^p=\Ker \delta$ is the space of $p$-cocycles, and $C_p=\Ker \partial$ is the space of $p$-cycles.

Now, define $D_2: C^k \rightarrow C_{n-k}$ to be the linear map
\[
\varphi \overset{D_2}{\longmapsto} \varphi \cap [M]
\]
where $[M]$ is a fundamental class for $M$.

In other words, $D_2$ assigns (up to sign) to each cocycle $\varphi \in C^k$ and every $n$-simplex $T=[v_0,\ldots,v_n]$ the last $(n-k)$-face of $T$, with multiplicity the value of $\varphi$ on the first $k$-face. 

Here are two examples, in dimension 2, of 1-cochains (first column), and their image under $D_2$ (second column).

\begin{figure}[H]
\centering
\includegraphics[width=90mm]{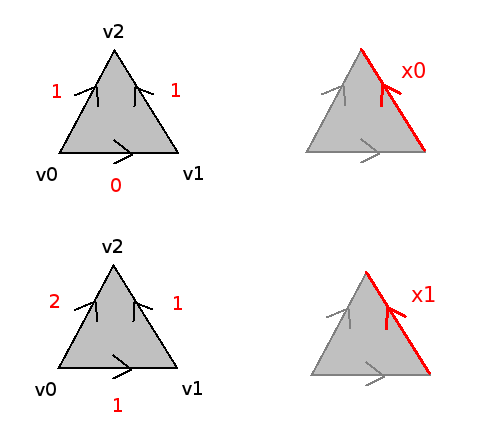}
\label{r2}
\end{figure}

Note that what $D_2$ does to every $n$-simplex is clear, since it forgets every face, except the last, and takes it with multiplicity the image of the first face under $\varphi$. Less clear is why that works, i.e. why we obtain an isomorphism; indeed, this process seem to forget a lot of information relative to $\varphi$, as we see in the previous picture.

The proof that $D_2$ is an isomorphism is usually done by using the bootstrap technique, see \cite{Hatcher}. We will attempt to give a more visual reason for that, in section \ref{isom}.

\mbox{}

Another natural question to ask is if $D_1$, defined in section \ref{cellulation}, and $D_2$ do the same work. It seems not, when comparing the respective images of the same cocycles (see section \ref{cellulation}). This matter will be elucidated in section \ref{deform} for $n=2$, and section \ref{generalization} for the general case.

%% file: levelcurves
In dimension 2, there is a nice way of seeing what a 1-cochain must verify to be a cocycle.
Namely, a 1-cochain $ \varphi $ is a cocycle if, and only if, the following holds for every simplex (triangle) $ [v_0,v_1,v_2] $:
\[ \varphi ([v_0,v_2])= \varphi ([v_0,v_1])+\varphi ([v_1,v_2]). \]
Interpreting every $ \varphi ([v_i,v_j]) $ as a difference of elevation, we see that this formula guarantees that we get to the same elevation by going directly from $ v_0 $ to $ v_2 $, or by going first from $ v_0 $ to $ v_1 $, and then from $ v_1 $ to $ v_2 $. 
Thus we can construct locally (by that we mean on restriction to any triangle) altitudes that are consistent with $ \varphi $, as in the following figure:

\begin{figure}[H]
\centering
\includegraphics[width=110mm]{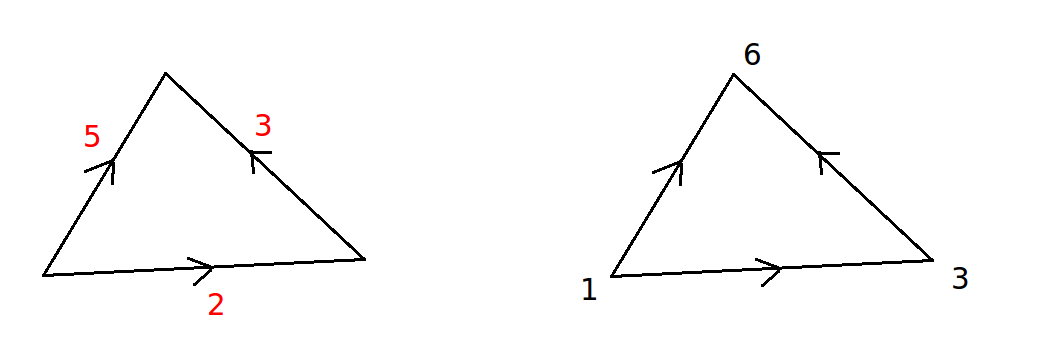}
\caption{A cocycle, represented in red, represents a difference of elevation, leading to altitudes on the vertices of the triangle, in black.}
\label{cn0}
\end{figure}

We would like to construct level curves for these local altitude, so that each triangle looks like a topographic map. That means that we are going to draw on the surface a system of distinct oriented loops, each corresponding to a constant elevation in our analogy, in such a way that the intersection number between the system of loops and each edge $ [v_i,v_j] $ is exactly $ \varphi( [v_i,v_j]) $. 

\begin{figure}[H]
\centering
\includegraphics[width=110mm]{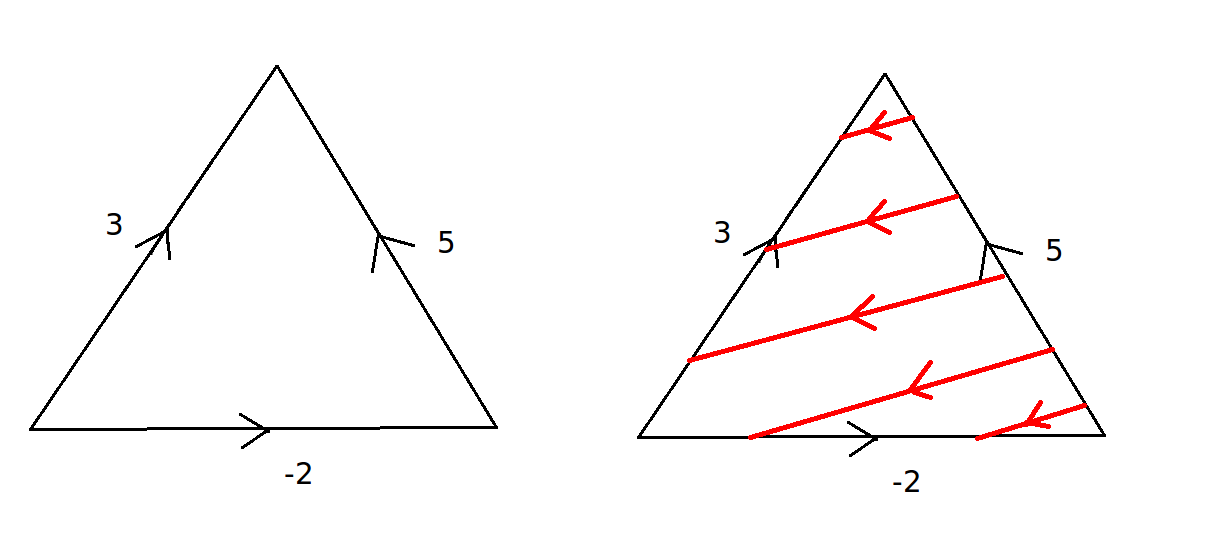}
\caption{Traveling on any edge $  [v_i,v_j] $, you cross exactly $ \varphi( [v_i,v_j]) $ level curves, counted with multiplicity.}
\label{ct0}
\end{figure}

For that purpose, let us restrict our attention to some triangle $ T=[v_0,v_1,v_2] $. We can always suppose that $ T $ is actually an euclidean triangle.
We can require the level curves to be piecewise linear, simply by thinking the vertices of the triangle to be of integer altitude,
 to extend this altitude linearly to the entire triangle, and then to take the level curves to be those lines on which the altitude is an integer plus a half, as shown on the next picture:

\begin{figure}[H]
\centering
\includegraphics[width=110mm]{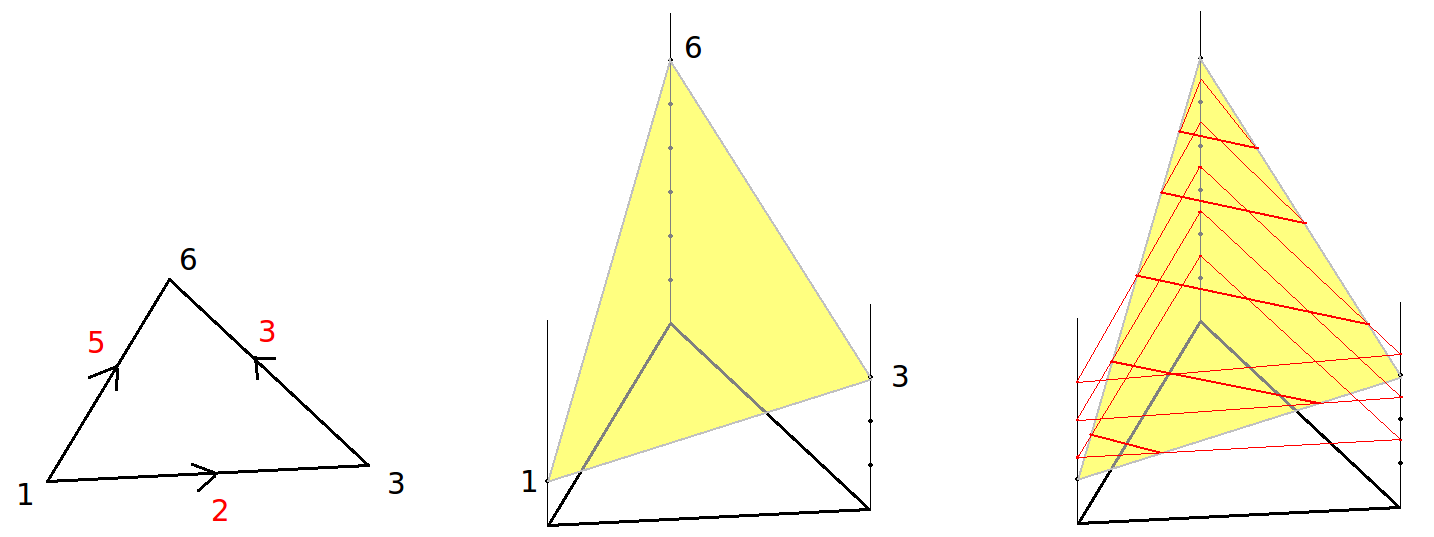}
\label{cn}
\end{figure}

Choosing the level curves to correspond to lines of elevation an integer plus a half, we see that if we go up 5 meters, say from altitude 1m to altitude 6m,  we cross exactly 5 curves.

Next, we orientate these curves using the following rule: the angle between the normal to a curve pointing to increasing elevation, and the direction of the curve itself should always be $ \pi / 2 $:

\begin{figure}[H]
\centering
\includegraphics[width=60mm]{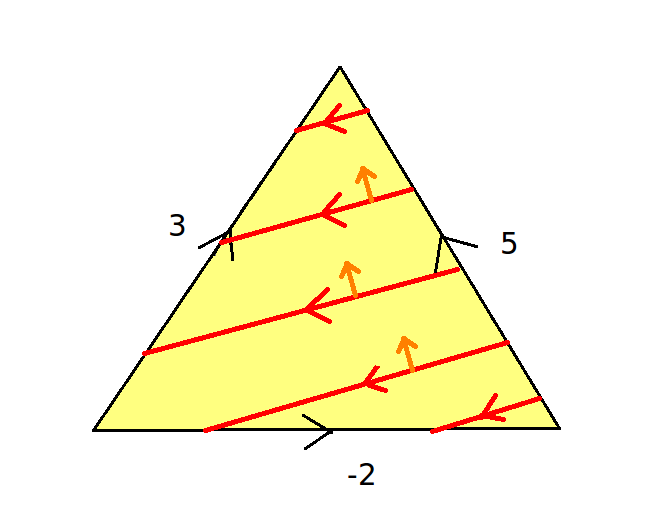}
\caption{On each level curve, the normal pointing to higher elevation, in orange, leads to an orientation of the level curve itself.}
\label{ct1}
\end{figure}

Note that if two triangles share the same edge, then the level curves of the first will end at the beginning of the level curves of the second, with consistent orientation:

\begin{figure}[H]
\centering
\includegraphics[width=60mm]{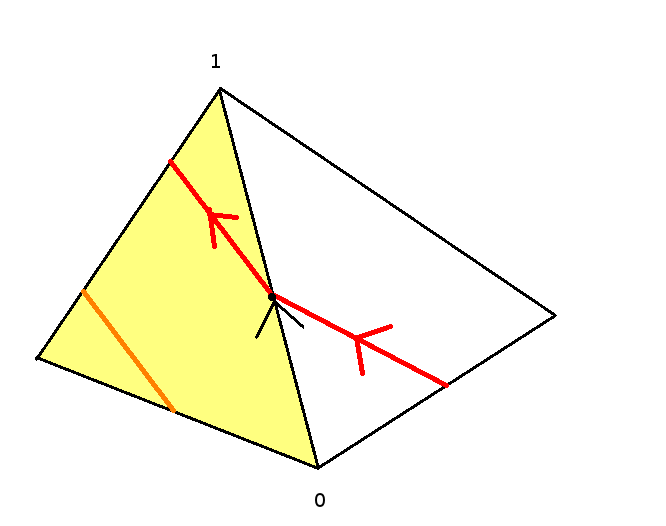}
\caption{Consistency of orientation of curves.}
\label{ct2}
\end{figure}

By assumption, there is a finite number of triangles in the triangulation, thus if we follow any level curve with the finger, we must come back to our starting point after entering and exiting finitely many triangles.

Thus the level curves of each triangle extend to distinct oriented loops on the surface, that we see as cycles in $ C_1 $. 

Therefore we can describe the whole process as a map $ D_3 : C^1 \rightarrow C_1 $. 

Note that at the moment we do not know what the addition of two level curves (as oriented loops) is.
The addition of these curves, considered as cycles in $ C_1 $, is well defined, but then it is unclear if the map $ D_3 $ is linear:

\begin{figure}[H]
\centering
\includegraphics[width=110mm]{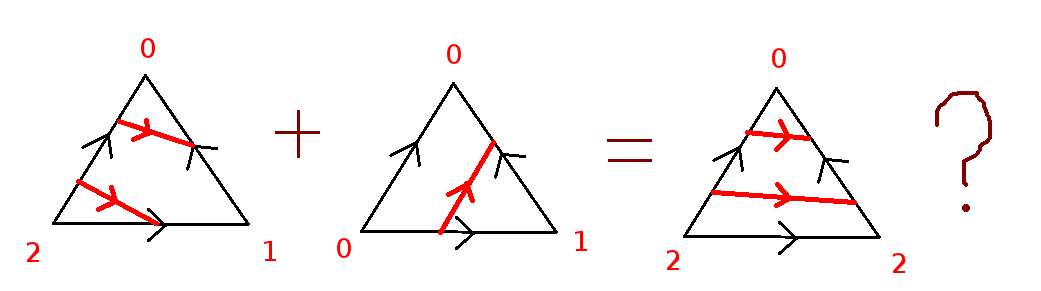}
\caption{Is $ D_3 $ a linear map?}
\label{sum}
\end{figure}

This matter will be settled in section \ref{deform}. 

%% file: deform
Now that we know three different ways to link cohomology and homology for manifolds, let us show that they actually do ``the same job''. 

We only consider here the case where $n=2$, and $k=1$, so that we can talk about level curves (cf. section \ref{levelcurves}). 

Thus, start with any 1-cocycle $\varphi$, and consider the level curves $D_3(\varphi)$ associated to it.  We are going to deform these level curves in two different ways, so as to get both $D_1(\varphi)$ and $D_2(\varphi)$, where $D_1$ and $D_2$ were defined in sections \ref{cellulation} and \ref{cap}.

First, in any triangle of the initial triangulation $\Delta$, consider the level curve consisting of some fixed constant altitude. Deform it by making its vertices glide along the edges of the triangle, in the direction of the highest index, like in the following picture:

\begin{figure}[H]
\centering
\includegraphics[width=110mm]{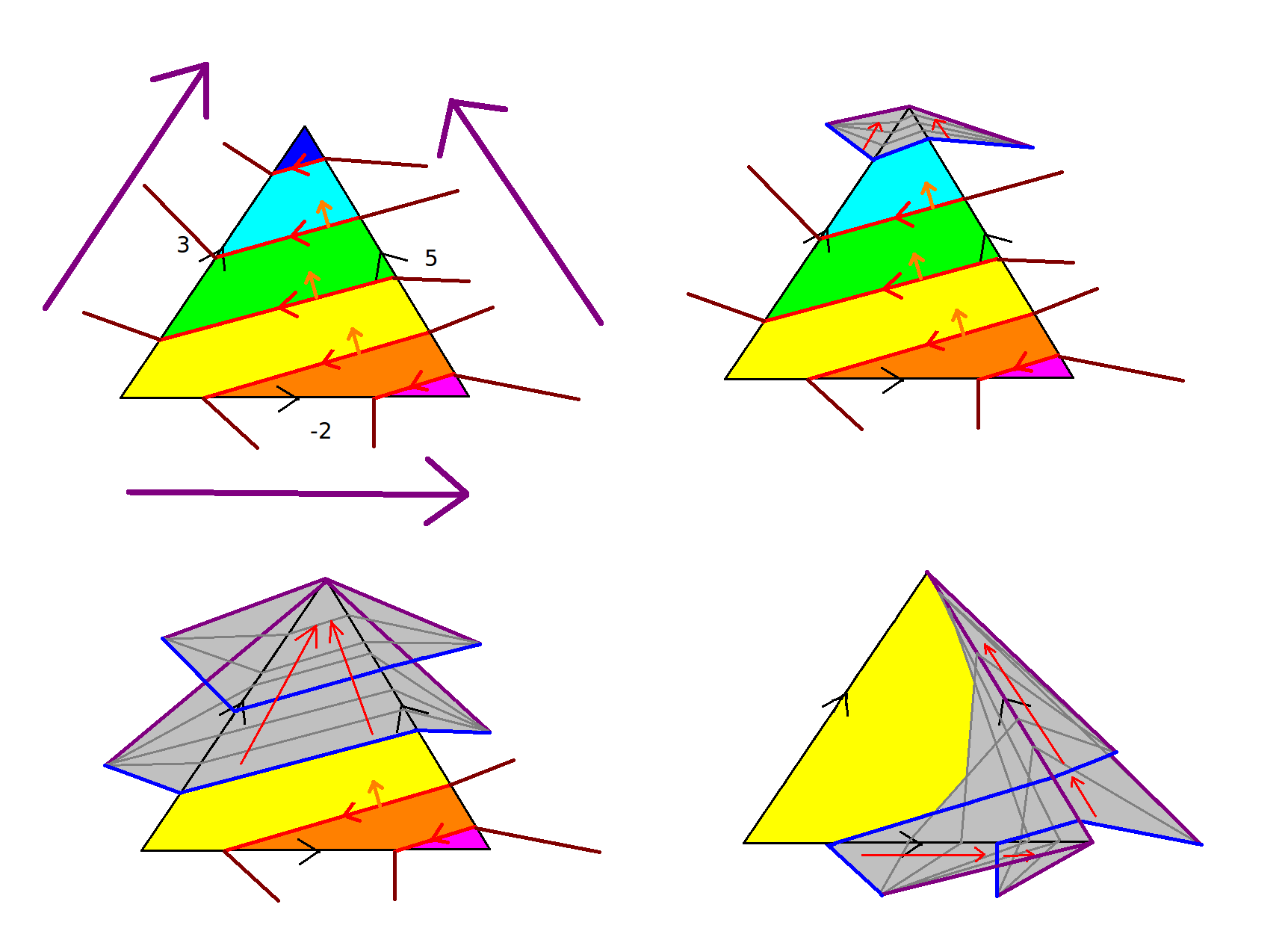}
\caption{Making vertices glide, in the direction of highest index.}
\label{ct13}
\end{figure}

Notice that we get here $-2$ times the vertex $ [v_1, v_2] $, which is exactly the cap product of the cocycle with the triangle $[v_0,v_1,v_2]$.

The same process can be undertaken for any triangle of the initial triangulation, as the following figure sketches:

\begin{figure}[H]
\centering
\includegraphics[width=110mm]{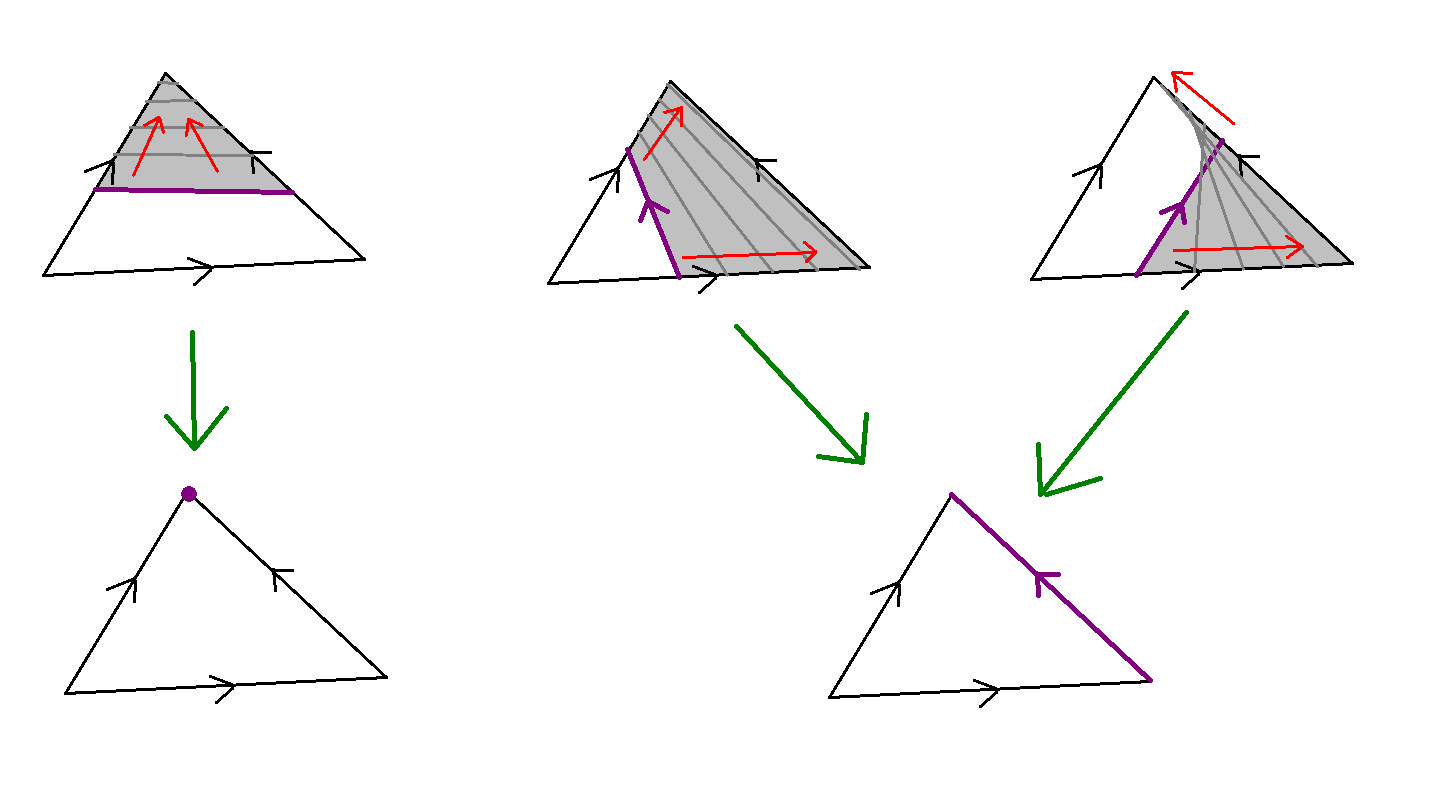}
\caption{Any edge joining $ [v_0, v_2] $ and $ [v_1, v_2] $ reduces to a point when its vertices glide toward the vertex of highest index $[v_3]$; any other edge is deformed into plus or minus $ [v_1, v_2] $.}
\label{cn10}
\end{figure}

As with see in figure \ref{cn10}, any edge joining $ [v_0, v_2] $ and $ [v_1, v_2] $ glides up toward the vertex $v_3$, thus degenerates into a point, and any other edge is deformed into $ [v_1, v_2] $ or its opposite, depending on the orientation of that edge. 
All in all, the result of that process on any triangle $ [v_1, v_2] $ is the edge $ [v_1, v_2] $, counted $ \varphi([v_0,v_1]) $ times. 

This result is exactly the cap product between $ \varphi $ and $ [v_1, v_2] $!

\mbox{}

Note that in the process, we only changed the curves by performing successive homotopies on them.
Thus, we did not change the class $D_3(\varphi)$, as regarded as an element of the homology $ H_1 = \Ker \partial / \im \partial $. 
Here of course, $H_1$ is thought to be a singular homology group.

\mbox{}

Similarly, we could make the level curves glide in the following way: first make every vertex of some level curve glide up to the center of the edge of the triangle that contains it; then make the center of the resulting edge glide freely within the triangle, up to the center of it, as the following three examples (to read vertically) show:

\begin{figure}[H]
\centering
\includegraphics[width=110mm]{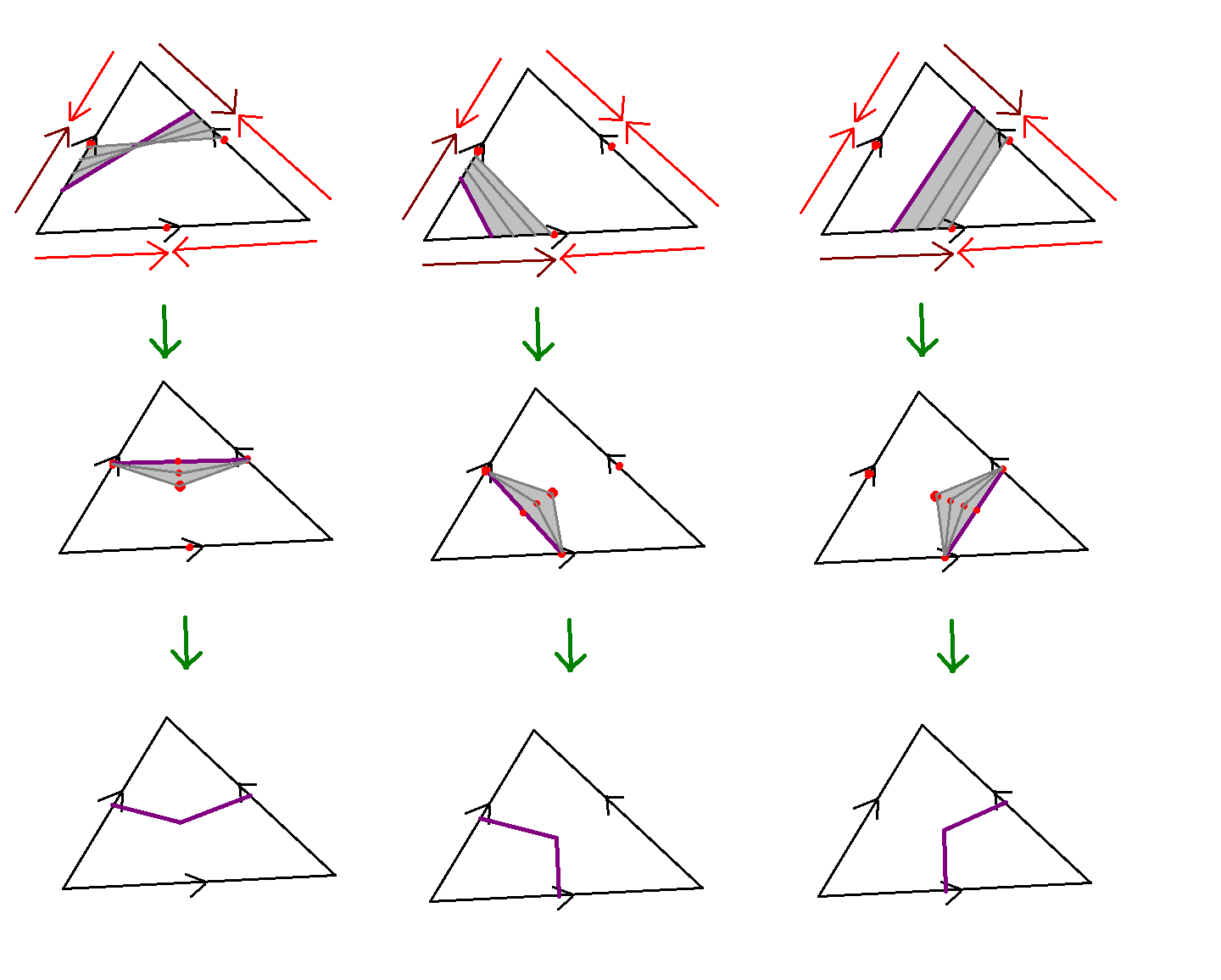}
\caption{Second way of making level curves glide.}
\label{cn11}
\end{figure}

This time, it is easy to see that the result of this process is exactly $D_1(\varphi)$.
As before, we did not change the class $D_3(\varphi)$, as we only applied homotopies to the curves.

We just proved the following proposition:
\begin{prop}
\label{prop}
$D_1$, $D_2$ and $D_3$ are equal, when regarded as mappings from $\Ker \delta$ to $H_1$.
\end{prop}

As a direct consequence, properties that are clear for one of these maps are clear for all. For instance, it is clear that $D_2$ is linear, so $D_3$ is linear also. It is clear, from a geometrical viewpoint, that $D_3(\varphi)$ is always a cycle, so whenever $\varphi$ is a cocycle, $D_1(\varphi)$ and $D_2(\varphi)$ are cycles. That fact is usually proven in a purely algebraic way.

%% file: generalization
In higher dimension, there is no direct analogue of level curves. Nevertheless, we can still deform the dual cellulation associated to a given cocycle $\varphi$ (i.e. $D_1(\varphi)$), so as to get the cap -product defined cycle $D_2(\varphi)$.

Formally, for any compact orientable triangulated manifold $M$ of dimension $n$ with triangulation $\Delta$, 
let $T=[v_0, \ldots, v_n]$ be a $n$-simplex of $\Delta$, and let $S=[v_{i_0}, \ldots, v_{i_k}]$ be some $k$-face of $T$, where as usual $i_1< \ldots <i_k$.

Recall that the dual cell associated to $S$ in $T$ is the convex hull of all barycenters of subsets of vertices of $T$ that contain the vertices of $S$.

For $n=3$, we get the following picture:

\begin{figure}[H]
\centering
\includegraphics[width=105mm]{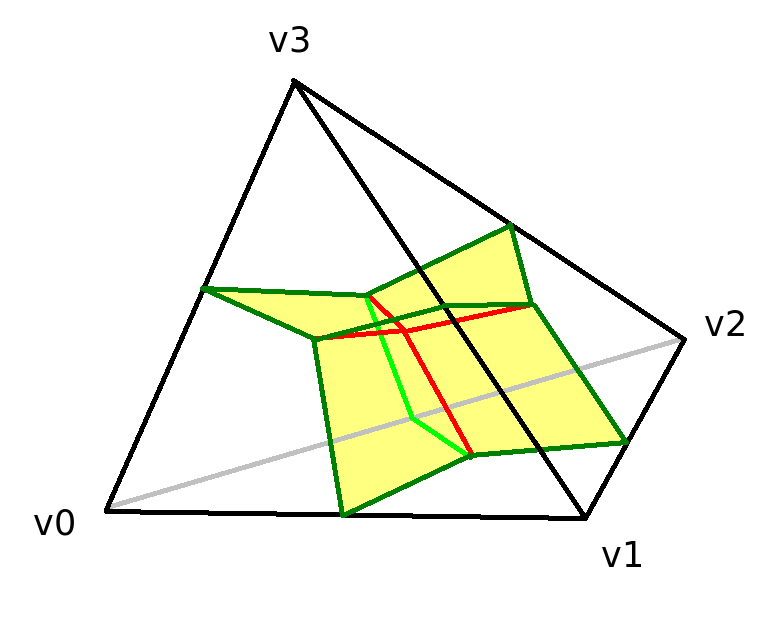}
\label{highdim2}
\end{figure}

\mbox{}

The way we are going to make these barycenters move is in direct analogy with what we did in dimension 2: make the barycenter of $\{ v_{j_0}, \ldots, v_{j_l}\}$ glide within the face $[v_{j_0}, \ldots, v_{j_l}]$, in the direction of the vertex of highest index.

In that process, the barycenter of every $\{ v_{j_0}, \ldots, v_{j_l} \}$ glides up to $v_{j_l}$.
Thus, the dual cell associated to $S=[v_{i_0}, \ldots, v_{i_k}]$ is easily seen to be deformed into $[v_{i_k}, v_{i_k + 1}, \ldots v_n]$.

That face is of dimension $n-i_k$, so it is either a genuine $(n-k)$-face if $i_k=k$ (i.e. $S=[v_0, \ldots, v_k]$), or it is degenerate (i.e. of dimension less than $n-k$) and will add up to nothing when we sum up every $k$-simplex of a fundamental class.

Thus, we can ignore these latter faces, and see that the only $(n-k)$-face left is $[v_k, \ldots, v_n]$, coming from the $k$-face $[v_0, \ldots, v_k]$.

Therefore, if you start with any cocycle $\varphi$ and restrict your attention to $T$, then make the dual cellulation of $\varphi$ move in the way explained above, you get the $(n-k)$-face $[v_k,\ldots,v_n]$, counted $\varphi([v_0,\ldots,v_k])$ times.

Thus as before, $D_1$ and $D_2$ are equal, when viewed as maps with image in $H_{n-k}$.

We believe that the 3-dimensional case is enlightening:

\begin{figure}[H]
\centering
\includegraphics[width=140mm]{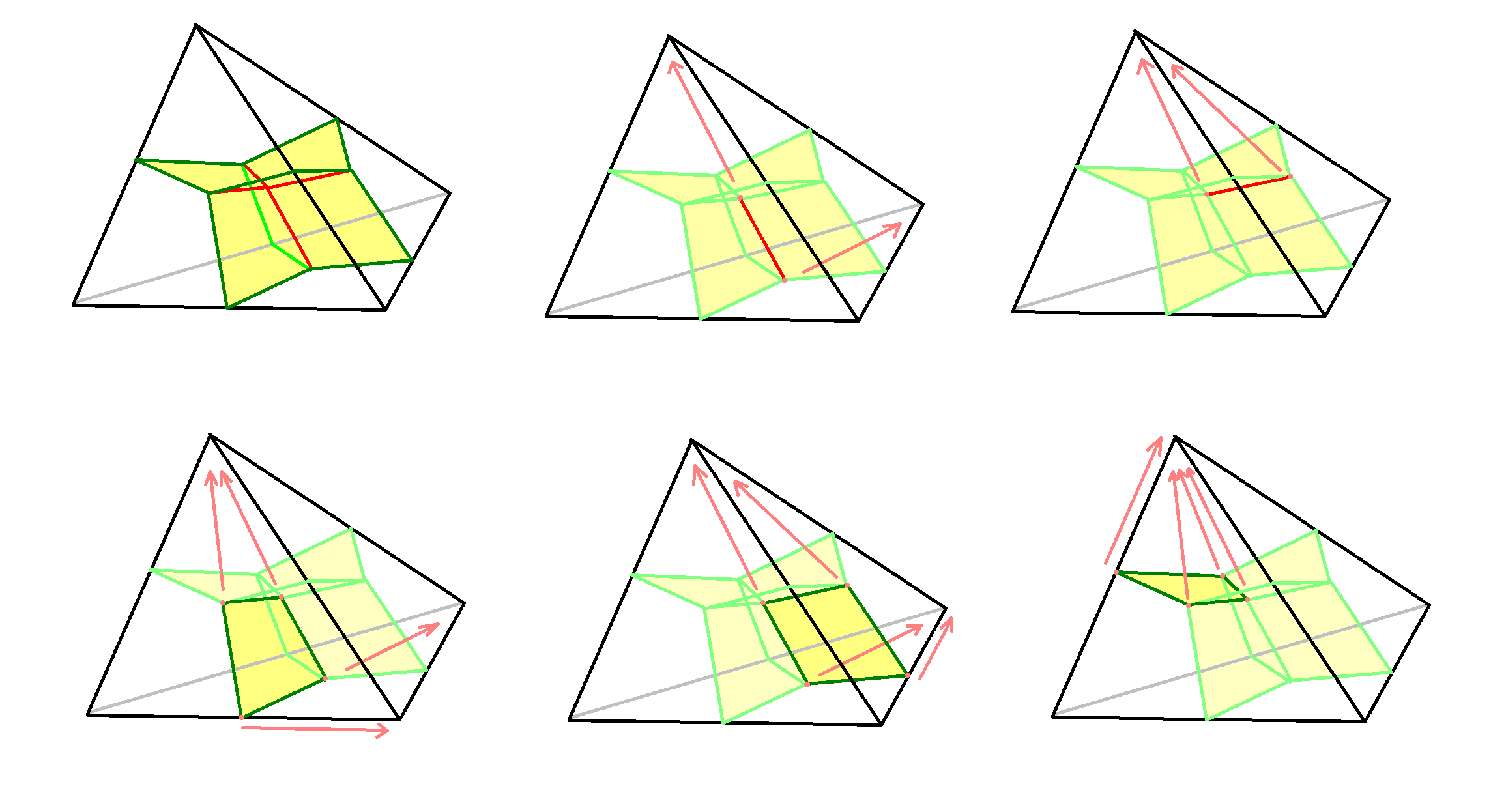}
\label{highdim}
\caption{The standard 3-simplex, its dual cellulation, and how to make vertices glide.}
\end{figure}

More precisely, what we see on the previous figure is:
\begin{itemize}
\item (from left to right, first row) The standard 3-simplex and its dual cellulation; the dual cell associated to $[v_0, v_1, v_2]$ glides up to $[v_2,v_3]$; the dual cell associated to $[v_1,v_2,v_3]$ degenerates into the point $v_3$;
\item (from left to right, second row) The dual cell associated to $[v_0,v_1]$ glides up to the face $[v_1,v_2,v_3]$; the dual cell associated to $[v_1,v_2]$ degenerates into the edge $[v_2,v_3]$ and the dual cell associated to $[v_0,v_3]$ degenerates into the point $v_3$.
\end{itemize}

%% file: isom
Trying to explain in a geometrical way why $D_1$ (or $D_2$) induces a map from $H^k $ to $H_{n-k}$ and why it is then an isomorphism, seems to be very close to doing the verifications directly with formulas, in an algebraic way. Thus, we will only show that, in the restrictive case where $n=2$, $D_3$ induces an isomorphism.

It is actually very natural to consider $D_3$ instead of $D_1$ or $D_2$ for that matter, since the latter are obtained by collapsing level curves (see section \ref{deform}); this is also motivated by Allen Hatcher's discussion about level curves in ``The idea of Cohomology'', in \cite{Hatcher}.

\mbox{}

To show that $D_3$ induces a map from $ H^1$ to $ H_1 $, we have to show that a coboundary is sent by $ C $ to a boundary. In fact, the converse is also true, as we see from the two following lemmas.

\begin{lemma}
\label{l1}
A cocycle $ \varphi $ is a coboundary if, and only if, the altitude-maps $ \psi_T $ (defined up to an integer) of every triangle $ T $ of the triangulation $ \Delta $ can be extended to form a global altitude $ \psi $, defined on the vertices of $\Delta$, which is consistent with $ \varphi $.
\end{lemma}
\begin{proof}
Saying that the altitude $ \psi $ is consistent with $\varphi$ means that for each edge $ [v_0, v_1] $, the difference of elevation between the two points is $ \varphi([v_0,v_1])= \psi(v_1)-\psi(v_0) $, i.e. it is saying that $\varphi=\delta \psi$.
\end{proof}

\begin{lemma}
A cocycle is a coboundary if, and only if the level curves associated to it are boundaries in $ H_1 $.
\end{lemma}
\begin{proof}
Direct direction: if $\varphi$ is a coboundary, then by lemma \ref{l1}, there is an altitude $\psi$ on the vertices of the triangulation, that is consistent with $\varphi$. Extend it linearly on the whole manifold $M$. Then by construction of $D_3$, every curve of  $D_3(\varphi)$ is of constant elevation $ \psi = n + 1/2 $ for some $n$. 
Write $ c_n $ for the curve of constant elevation $ \psi = n + 1/2 $.

It is easy to see that the topological boundary of each $ \alpha_n= \{ x \in M,\; \psi(x) \leq n+ 1/2\} $ is the non-oriented loop which is the trace of $ c_n $.
We leave to the reader to see how to tile the $ \alpha_n $ with signed triangles so that they become 2-chains in $ H_2 $ with boundary $ c_n \in H_1 $.
Then the image of $ \varphi $ under $ D_3 $ is by construction $ \sum_n c_n $, that is $ \sum_n \partial \alpha_n $, which is obviously a boundary.

\begin{figure}[H]
\centering
\includegraphics[width=110mm]{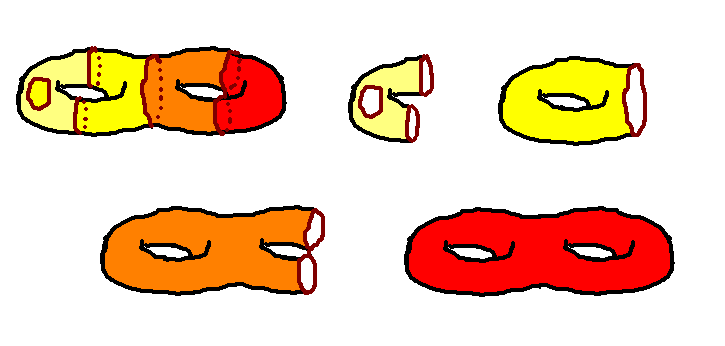}
\caption{We can always tile the colored regions $\alpha_n$ with triangles, so that their boundary is~$ c_n $.}
\label{torus6}
\end{figure}

Indirect direction: suppose $D_3(\varphi)$ is a boundary. Then, notice it is a boundary of a chain in $H_2$ which is the sum of triangles of some subdivision $\widetilde{\Delta}$ of $\Delta$. 
Let us write $ T_i $ for the triangles of $\widetilde{\Delta} $, signed so that $ [M]=\sum_i T_i $ is a fundamental class of $M$.
Then we can write $ \beta  $ as the sum $ \beta=\sum_i \lambda_i T_i $ with $ \lambda_i $ integers.

After examination of the condition that $ \partial \beta $ is the image of $ \varphi $ under $ D_3 $, we see by induction on adjacent triangles of $\Delta$ that $ \lambda_i = \lambda_j $ for every triangles $T_i$, $T_j$ on the same topological connected component of $ M $ minus the trace of the level curves:

\begin{figure}[H]
\centering
\includegraphics[width=85mm]{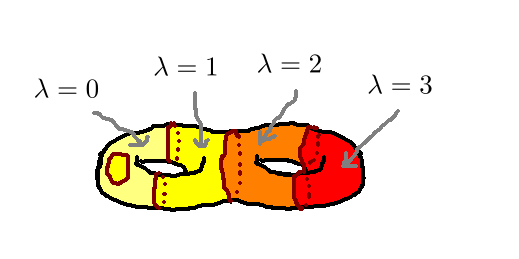}
\label{torus7}
\end{figure}

Since every vertex of $\Delta$ is not on the trace of any level curve, we can define $\psi$ to assign to every vertex $[v]$ of $\Delta$ the integer $\lambda_i$, where the corresponding triangle $T_i$ contains $[v]$. By the preceding remark, this does not depend on the choice of a triangle containing $[v]$.

We then obtain the required altitude; details are left to the reader.
\end{proof}

We obtain, as a direct corollary:
\begin{cor}
\label{cor}
$D_3$ induces an injective linear map $ H^1 \mapsto H_1$, that we still call $D_3$.
\end{cor}

To show $D_3$ is an isomorphism, we only have left to show that it is surjective, which is the purpose of the next lemma.

\begin{lemma}
\label{surj11}
$D_3$ is surjective.
\end{lemma}
\begin{proof}
Start with any cycle of $H_1$. We can in fact assume it is represented by a simple closed loop.
Thicken it to a band of triangles, as in the following figure (where we have to think of the right and left parts of the band as being glued together):

\begin{figure}[H]
\centering
\includegraphics[width=110mm]{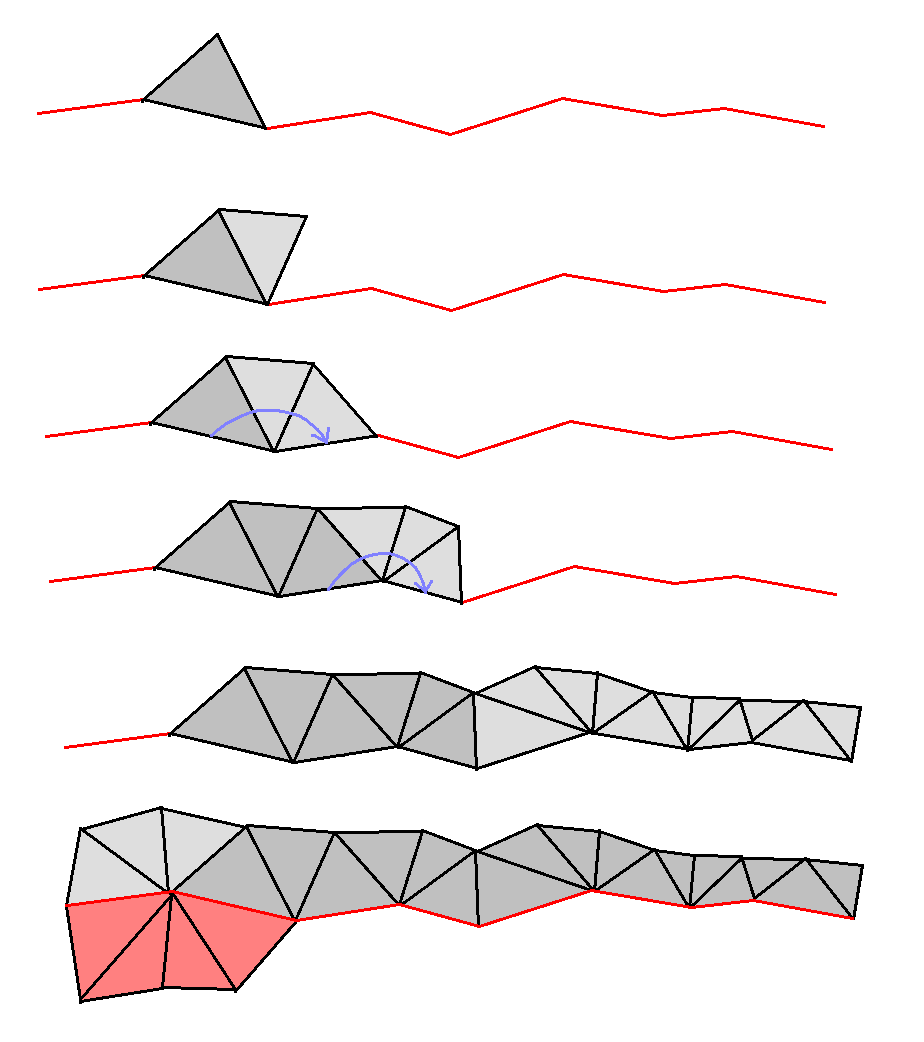}
\label{surj}
\end{figure}

The case in pink does not appear, since $M$ is orientable. 

Perform a perturbation on the loop, in order to get a loop transversal to the triangulation (in pink in the following figure)

\begin{figure}[H]
\centering
\includegraphics[width=110mm]{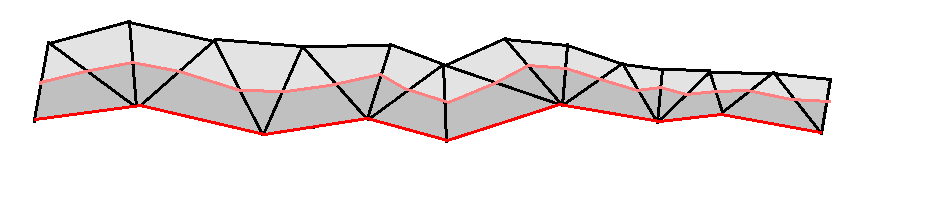}
\label{surj2}
\end{figure}

That loop is clearly the level curve of some cocycle, which is $\pm 1$ on every edge of $\Delta$ that the pink loop crosses, and 0 elsewhere (details about orientation are left to the reader).
Since it is homotopic to the red cycle, we are done.
\end{proof}

\begin{cor}
In dimension 2, $D_1$, $D_2$ and $D_3$ are all isomorphisms.
\end{cor}
\begin{proof}
Combine proposition \ref{prop}, corollary \ref{cor} and lemma \ref{surj11}. Notice that everything was performed geometrically.
\end{proof}

%% file: nonor
Notice that in everything we made insofar, we used extensively the ordering of the vertices of the triangulation; the actual orientability of $M$ is only used to provide signs.

As a result, by choosing any ordering of the vertices, and ignoring orientation issues, every construction and argument in this paper is directly transposable to the non-orientable case, except for the last lemma about the surjectivity of $D_3$, which can be fixed by a minor change, that we leave to the reader.